\documentclass[reqno]{amsart}
\usepackage{amsthm,amsmath}
\usepackage{amssymb}
\usepackage[colorlinks]{hyperref}
\usepackage{amsfonts}
\usepackage{color}

    \newtheorem{Def}{Definition}
	\numberwithin{equation}{section}
	\newtheorem{The}{Theorem}[section]
	\newtheorem{Lem}[The]{Lemma}
	\newtheorem{Pro}[The]{Proposition}
	
	\newtheorem{example}[The]{Example}
	\newtheorem{corollary}[The]{Corollary}

\begin{document}
\bibliographystyle{plain}

\title[Fourier integral operators with weighted symbols]{Fourier integral operators with weighted symbols}

\author[O. Elong]{Elong Ouissam }
\author[ A. Senoussaoui, ]
{Senoussaoui Abderrahmane }  % in alphabetical order

\address{ELONG Ouissam \newline
	Laboratory of Fundamental and Applied Mathematics of Oran
	‘LMFAO’, Department of Mathematics, University of Oran Ahmed Benbella, Oran, Algeria}
\email{elong\_ouissam@yahoo.fr}

\address{SENOUSSAOUI Abderrahmane\newline
	Laboratory of Fundamental and Applied Mathematics of Oran
	‘LMFAO’, Department of Mathematics, University of Oran Ahmed Benbella, Oran, Algeria}
\email{senoussaoui\_abdou@yahoo.fr or senoussaoui.abderahmane@univ-oran.dz}

\date{}
%\thanks{Submitted June 10, 2005. Published March 9, 2006.}
\subjclass[2010]{35S30, 35S05, 47G30}
\keywords{Fourier integral operators, pseudodifferential operators,
	\hfill\break\indent symbol and phase, boundedness and compactness}

\begin{abstract}
	 The paper contains a survey of a class of Fourier integral operators defined by symbols with tempered weight.
	 These operators are bounded (respectively compact) in  $L^2$ if the weight of the amplitude is bounded (respectively tends to $0$).
\end{abstract}

\maketitle
	
	\newcommand{\Om}{\Omega }
	\newcommand{\eps}{\varepsilon}
	\newcommand{\al}{\alpha}
	\newcommand{\lm}{\lambda}
	\newcommand{\ls}{<}
	\newcommand{\gr}{>}
	\newcommand{\ra}{\rightarrow}
	\newcommand{\ov}{\overline}
	\newcommand{\pr}{\partial}
	\newcommand{\wt}{\tilde}
	\newcommand{\n}{\mathbb{N}}
	\newcommand{\rr}{\mathbb{R}}
	\newcommand{\cc}{\mathbb{C}}
	\newcommand{\rn}{\mathbb{R}^n}
	\newcommand{\rN}{\mathbb{R}^N}
	\newcommand{\nn}{\mathbb{N}^n}
	\newcommand{\nN}{\mathbb{N}^N}
	\newcommand{\s}{\mathcal{S}}
	\newcommand{\ps}{\mathcal{S^{\prime}}}
	\newcommand{\lp}{L^p}
	\newcommand{\dd}{\mathcal{D}}
	\newcommand{\f}{\times}

\section{Introduction}
	A Fourier integral operator or FIO for short has the following form
	\begin{equation}\label{eq1}
	[I(a,\phi)f](x)=\underset{\rn_y\f\rN_\theta}\iint e^{i\phi(x,y,\theta)}a(x,y,\theta)f(y)\,dy\,d\theta,\;\;\;\; f\in\s(\rn)
	\end{equation} where $\phi$ is called the \textit{phase function} and $a$ is the \textit{symbol} of the FIO $I(a,\phi)$.
	
	The study of FIO was started by considering the well known class of symbols $S^m_{\rho,\delta}$ introduced by H\"ormander which consists of functions $a(x,\theta)\in C^\infty(\rn\f\rN)$ that satisfy $$|\pr^\al_\theta\, \pr^\beta_x a(x,\theta)\leq C_{\al,\beta}(1+|\theta|)^{m-\rho|\al|+\delta|\beta|},$$ with $m\in\rr,\; 0\leq\rho,\delta\leq1$. For the phase function one usually assumes that $\phi(x,\theta)\in C^\infty(\rn\f\rN\setminus0)$ is positively homogeneous of degree $1$ with respect to $\theta$ and $\phi$ does not have critical points for $\theta\neq0$.
	
	Later on, other classes of symbols and phase functions were studied. In  (\cite{Helffer}) and \cite{Robert}, D. Robert  and B. Helffer treated the symbol class $\Gamma^\mu_\rho(\Om)$ that consists of elements $a\in C^\infty(\Om)$ such that for any multi-indices $(\al,\beta,\gamma)\in \nn\f\nn\f\nN,$ there exists  $C_{\al,\beta,\gamma}>0,$ $$|\pr^\al_x\pr^\beta_y\, \pr^\gamma_\theta a(x,y,\theta)\leq C_{\al,\beta,\gamma}\lm^{\mu-\rho(|\al|+|\beta|+|\gamma|)}(x,y,\theta),$$ where $\Om$ is an open set of $\rn\f\rn\f\rN,\, \mu\in \rr$ and $\rho\in[0,1]$ and they considered phase functions satisfying certain properties. In (\cite{Me-Se}), Messirdi and Senoussaoui treated	the $L^2$ boundedness and $L^2$ compactness of FIO with symbol class just defined. These operators are continuous (respectively compact) in  $L^2$ if the weight of the symbol is bounded (respectively tends to $0$). Noted that in H\"ormander's class this result is not true in general. In fact, in (\cite{Se-un_FIO}) the author gave an example of FIO with symbol belonging to $\underset{0<\rho<1}\bigcap S^0_{\rho,1}$ that cannot be extended as a bounded operator on $L^2(\rn)$.
	
	The aim of this work is to extend results obtained in (\cite{Me-Se}), we save hypothesis on the phase function but we consider symbols with weight $(m,\rho)$ (see below).
	
	So in the second section we define symbol and phase functions used in this paper and we give the sense of the integral (\ref{eq1}) by using the known oscillatory integral method developed by H\"ormander.  A special case of phase functions treated here is discussed in the preliminaries, in the third section. The last section is devoted to treat the $L^2$ boundedness and $L^2$ compactness of FIO.

\section{Preliminaries}
  \begin{Def}
  	A continuous function $m:\mathbb{R}^{n}\rightarrow[0,+\infty[ $ is called  a tempered weight on  $\mathbb{R}^{n}$ if   	
  	$$\exists C_{0}>0,\text{ }\exists l\in \mathbb{R}\;;\text{ }m(x) \leq C_{0}\,m\left( x_{1}\right) \left( 1+\left\vert x_{1}-x\right\vert	\right) ^{l},\;\forall x,x_{1}\in \mathbb{R}^{n}.$$
  	
  	\end{Def}
    	Functions of the form $\lambda ^{p}(x)=( 1+|x|) ^{p},$ $p\in \mathbb{R}{,}$ define tempered weights.
  	\begin{Def}
  		Let $\Om $ be an open set in $\mathbb{R}^{n}$, $\rho \in[0,1]$ and $m$ a tempered weight. A function $a\in C^{\infty }\left( \Om \right)$ is called symbol with weight $(m,\rho )$ or $(m,\rho )$-weighted symbol on $\Om$ if   		  		
  		$$\forall \alpha \in \n^{n},\exists C_{\al}>0;\;\;\left\vert\pr _{x}^{\al}a\left( x\right) \right\vert \leq C_{\al}m\left(x\right) \left( 1+\left\vert x\right\vert \right) ^{-\rho \left\vert \al\right\vert },\;\forall x\in \Om.$$
  		\end{Def}
  		We note  $S_{\rho }^{m}(\Om) $ the space of symbols with weight $(m,\rho )$.
  		% and instead of $S_{\rho }^{m}(\Om)$, we will simply write $S_{\rho }^{m}$.
  	\begin{Pro} Let $m$ and $l$ be two tempered weights.
  		\begin{enumerate}
  				\item[($i$)] If $a\in S^m_\rho$ then $\pr ^\al_x\pr^\beta_\theta a\in S^{m\lm^{-\rho|\al+\beta|}}_\rho$;
  				\item[($ii$)] If $a\in S^m_\rho$ and $b\in S^l_\rho$ then $ab\in S^{ml}_\rho$;
  				\item[($iii$)] If $\rho\leq \delta$, $S^m_\delta\subset S^m_\rho$;
  				\item[($iv$)] Let $a\in S^m_\rho$. If there exists $C_0\gr0$ and $\mu\in \rr$ such that $|a|\geq C_0 \lm^{\mu}$ uniformly on $\Om$ then $\frac{1}{a}\in S^{m\lm^{-2\mu}}_\rho$.\\  				
  		\end{enumerate}
  	\end{Pro}
  	\begin{proof}
  		For the proof we use Leibniz formula.
  		($ii$) is obtained by Leibniz formula and by induction we prove ($iv$).
  	\end{proof}			
 	Now, we consider the class of Fourier integral operators
  	\begin{equation}\label{fio}
  			[I(a,\phi)f](x)=\underset{\rn_y\f\rN_\theta}\iint e^{i\phi(x,y,\theta)}a(x,y,\theta)f(y)\,dy\,\widehat{d\theta },\;\;\;\; f\in\s(\rn)
  	\end{equation}
  	where $\widehat{d\theta }=(2\pi )^{-n}\,d\theta$, $a\in S^m_\rho$ and $\phi$ be a phase function which satisfies the following hypothesis
  		\begin{enumerate}
  			\item[(H1)]  $\phi \in C^{\infty }(\rr_{x}^n\times \rr_{y}^n\times \rr_{\theta }^{N},\mathbb{R})$ 	 ($\phi$ is a real function);                					
  			\item[(H2)]  For all $(\alpha ,\beta ,\gamma )\in \mathbb{N}^n\times \mathbb{N}^n\times \mathbb{N}^{N}$, there exists $C_{\alpha ,\beta,\gamma }>0$ such that
  				\begin{equation*}
  					|\pr_{x}^{\al}\pr _{y }^{\beta}\pr _{\theta}^{\gamma }\phi (x,y,\theta)|
  					\leq C_{\al,\beta ,\gamma }\lambda ^{(2-|\alpha | -|\beta |-|\gamma | )}(x,y,\theta);
  				\end{equation*}
  			\item[(H3)] There exists $K_{1},K_2>0$ such that
  				\begin{equation*}
  					K_{1}\lm (x,y,\theta)\leq \lm (\pr _{y}\phi ,\pr _{\theta }\phi ,y)\leq K_2\lm(x,y,\theta),\quad\forall (x,y,\theta)\in \rr_{x}^n\times \rr_{y}^n\times\rr_{\theta }^{N};
  				\end{equation*}
  			\item[(H3$^*$)] There exists $K_{1}^*,K_2^*>0$ such that
  			\begin{equation*}
  			K_{1}^*\lm (x,y,\theta)\leq \lm (x,\pr_\theta\phi,\pr_x\phi)\leq K_2^*\lm(x,y,\theta),\quad\forall (x,y,\theta)\in \rr_{x}^n\times \rr_{y}^n\times\rr_{\theta }^{N}.
  			\end{equation*}
  		\end{enumerate}

  	To give a meaning to the right hand side of (\ref{fio})	 we use the oscillatory integral method. So we consider $g\in \s(\rr_{x}^n\times \rr_{y}^n\times\rr_{\theta }^{N})$ such that $g(0)=1$. Let $a\in S^m_0$, we define $$a_\sigma(x,y,\theta)=g(\frac{x}{\sigma},\frac{y}{\sigma},\frac{\theta}{\sigma})\, a(x,y,\theta),\quad \sigma>0.$$
  	\begin{The}\label{existence}
  		Let $a\in S^m_0(\rn)$ and $\phi$ be a phase function which satisfies $(H1)-(H3)$. Then 
  		\begin{enumerate}
  			\item For all $f\in \mathcal{S}(\mathbb{R}^n)$,
  			$\underset{\sigma \to +\infty }\lim [ I(a_{\sigma },\phi)f] (x)$
  			exists for every point
  			$x\in \mathbb{R}^n$ and is independent of the choice of
  			the function $g$. We define
  			\begin{equation*}
  			[I(a,\phi )f ](x):=\lim_{\sigma \to +\infty}
  			[I(a_{\sigma },\phi )f ](x);
  			\end{equation*}
  			\item $I(a,\phi)$ defines a linear continuous operator on $\s(\rn)$ and $\s'(\rn)$ respectively.
  
  		\end{enumerate}
  	\end{The}
  						  	
  	\begin{proof}  						
  		Let $\chi\in C^\infty_{0}(\rr)$, $supp\, \chi\subset [-2,2]$ such that $\chi\equiv1$ in $[-1,1]$. For $\eps\gr0$, put $$\omega_\eps(x,y,\theta)=\chi\left(\frac{|\nabla_y\phi|^2+|\nabla_\theta\phi|^2}{\eps\, \lm^2(x,y,\theta)}\right). $$  						
  		$\bullet$ In $supp\,\omega_\eps$, $|\nabla_y\phi|^2+|\nabla_\theta\phi|^2\leq 2\eps\, \lm^2(x,y,\theta)$.
  		Using $(H3)$ we have $$K_1^2\, \lm^2(x,y,\theta)\leq 2\eps \lm^2(x,y,\theta)+|y|^2$$
  	    So for $\eps$ sufficiently small, fixed at value $\eps_0$, we obtain
  		\begin{equation}\label{majoration}
  		\lm^2(x,y,\theta)\leq C(\eps) |y|^2,\;\;\; \forall(x,y,\theta)\in supp\,\omega_{\eps_0}.
  		\end{equation}
%		Hence  			\begin{equation}
  				%\lm(x,y,\theta)\leq\sqrt{ C(\eps)} |y|,\;\;\; \forall(x,y,\theta)\in supp\,\omega_{\eps_0}.
	  			 %\end{equation}
  		Consequently \begin{eqnarray*}
  			|I(\omega_{\eps_0}a_\sigma,\phi)f(x)| & \leq &\underset{\rn_y\times\rN_\theta}\iint  |a_\sigma(x,y,\theta)| \, |f(y)| \,dy\, \widehat{d\theta }\\
  							     & \leq &\underset{\rn_y\times\rN_\theta}\iint  m(x,y,\theta)\, |f(y)| \,dy\, \widehat{d\theta }
  				              \end{eqnarray*}
  		Using the definition of the tempered weight, there exists $ C\gr0$ and $l\in \rr$ such that  
  		$$|I(\omega_{\eps_0}a_\sigma,\phi)f(x)| \leq C\, m(0,0,0)\underset{\rn_y\times\rN_\theta}\iint  \lm^l(x,y,\theta)\, |f(y)| dy\, \widehat{d\theta }.$$  	 
  	    and since $f\in \s(\rn)$, we deduce that $I(\omega_{\eps_0}a_\sigma,\phi)f $ is absolutely convergent on $supp\,\omega_\eps$. By the Lebesgue's dominated convergence theorem we can see easily that $$	I(\omega_{\eps_0}a,\phi )f =\lim_{\sigma \to +\infty}
  	    I(\omega_{\eps_0}a_{\sigma },\phi )f .$$ 
  							
  		$\bullet$ In $supp\,(1-\omega_{\eps_0})$, we have
  		$$  supp\,(1-\omega_{\eps_0})\subset\Om_0=\{(x,y,\theta):|\nabla_y\phi|^2+|\nabla_\theta\phi|^2\geq \eps_0\, \lm^2(x,y,\theta)  \}$$
  		Consider the differential operator 
  			$$L=\frac{1}{i(|\nabla_y\phi|^2+|\nabla_\theta\phi|^2 )}\left(  \sum_{j=1}^{n}\frac{\pr\phi}{\pr y_j}\frac{\pr}{\pr y_j}+\sum_{k=1}^{N}\frac{\pr\phi}{\pr \theta_k}\frac{\pr }{\pr \theta_k}\right).$$
  		A basic calculus shows that $L \; e^{i\phi}=e^{i\phi}$.
  							
   \begin{Lem}\label{lem1}
  		For any $b\in C^\infty(\rn_y\times\rN_\theta)$ and any $k\in \n$ we have
  		$$(^tL)^k[(1-\omega_{\eps_0})b]=\sum_{|\al|+|\beta|\leq k} g^{(k)}_{\al\beta}\pr^\al_y\pr^\beta_\theta((1-\omega_{\eps_0})b)$$ where $^tL$ is the transpose operator of $L$ and $g^{(k)}_{\al,\beta}\in S^{\lm^{-k}}_0(\Om_0)$.
  \end{Lem}
  								
  \begin{proof}
  		 The transpose operator $^tL$ has the following form
  				$$^tL=\sum_{j=1}^{n}F_j\,\frac{\pr}{\pr y_j}+\sum_{j=1}^{N}G_j\,\frac{\pr}{\pr \theta_j}+H,$$
  	  where \begin{eqnarray*}
  				F_j&=&-\frac{1}{i(|\nabla_y\phi|^2+|\nabla_\theta\phi|^2 )}\frac{\pr\phi}{\pr y_j},\\
  				G_j&=&-\frac{1}{i(|\nabla_y\phi|^2+|\nabla_\theta\phi|^2 )}\frac{\pr\phi}{\pr\theta_j},\\
  					H&=&-\frac{1}{i(|\nabla_y\phi|^2+|\nabla_\theta\phi|^2 )}\left( \sum_{j=1}^{n}\frac{\pr^2\phi}{\pr y_j^2}+\sum_{k=1}^{N}\frac{\pr^2\phi}{\pr \theta_k^2}\right). \\					
  		\end{eqnarray*}
  	 In $\Om_0$, $|\nabla_y\phi|^2+|\nabla_\theta\phi|^2\geq \eps_0\, \lm^2(x,y,\theta)  \}$ therefore using hypothesis $(H2)$ we find
  	  \begin{center}
	  	$F_j,G_j\in S^{\lm^{-1}}_0(\Om_0)$ and $H\in S^{\lm^{-2}}_0(\Om_0)$.
  	\end{center}
  	The lemma is deduced by induction on $k$. 
  \end{proof}
  \begin{equation}\label{int}
  			I((1-\omega_{\eps_0})a_\sigma,\phi)f(x)=\underset{\rn_y\times\rN_\theta}\iint  e^{i\phi(x,y,\theta)}(^tL)^k[(1-\omega_{\eps_0})\,a_\sigma f(y)]\, dy\, \widehat{d\theta }.
 \end{equation}
  Consequently, for $k$ large enough, the integral (\ref{int}) converges when $\sigma\rightarrow0$ to the absolutely convergent integral $$\underset{\rn_y\times\rN_\theta}\iint  e^{i\phi(x,y,\theta)}(^tL)^k[(1-\omega_{\eps_0})\,a f(y)]\, dy\, \widehat{d\theta }.$$
  To prove the second part, we use again the lemma (\ref{lem1}).
 \end{proof}

\section{Preliminaries}
 In the sequel we study the special phase function 
		\begin{equation}
			\phi(x,y,\theta)=S(x,\theta)-y\theta.\label{phase_S}
		\end{equation}
	where $S$ satisfies
	\begin{itemize}
		\item[(G1)] $S\in C^{\infty }(\rn_{x}\times
		\rn_{\theta },\rr)$,
		
		\item[(G2)] There exists $\delta _{0}>0$ such that
		\begin{equation*}
		\inf_{x,\theta \in \mathbb{R}^n}|\det \frac{\partial ^2S}{
			\partial x\partial \theta }(x,\theta )| \geq \delta
		_{0}.
		\end{equation*}
		
		\item[(G3)] For all $(\alpha ,\beta )\in \mathbb{N}
		^n\times \mathbb{N}^n$, there exist
		$ C_{\alpha ,\beta }>0$, such that
		\[
		|\partial _{x}^{\alpha }\partial _{\theta }^{\beta
		}S(x,\theta )| \leq C_{\alpha ,\beta }\lambda (x,\theta
		)^{(2-|\alpha | -|\beta |
			)}.
		\]
	\end{itemize}
	
\begin{Lem}
	If $S$ satisfies (G1), (G2)  and (G3), then
	$S$ satisfies (H1), (H2) and (H3). Also
	there exists $ C >0$ such that for all
	$x,x',\theta\in \rn$,
	\begin{equation}
	|x-x'| \leq C|(\partial _{\theta }S)(x,\theta )-(\partial _{\theta
	}S)(x',\theta )| .  \label{3.3}
	\end{equation}
\end{Lem}
   \begin{Pro} 
   	If $S$ satisfies (G1) and
   	(G3),  then there exists a constant
   	$\epsilon_{0}>0$  such that the phase function
   	$\phi$ given in \eqref{phase_S}
   	belongs to $S_{1}^{\lambda^2}(\Omega _{\phi ,\epsilon _{0}})$ where
   	\begin{equation*}
   	\Omega _{\phi ,\epsilon _{0}}
   	=\big\{ (x,\theta ,y)\in \mathbb{R}
   	^{3n};\;\;|\partial _{\theta }S(x,\theta )
   	-y| ^2<\epsilon _{0}\;(|
   	x| ^2+|y| ^2+|\theta | ^2)\big\} .
   	\end{equation*}
   \end{Pro}
   
   \begin{proof}
   	We have to show that: There exists $\epsilon _{0}>0$, such that for
   	all $\alpha ,\beta ,\gamma \in \mathbb{N}^n$,
   	there exist $C_{\alpha ,\beta ,\gamma }>0$:
   	\begin{equation}
   	|\partial _{x}^{\alpha }\partial _{y}^{\beta
   	}\partial _{\theta }^{\gamma }\phi (x,y,\theta )| \leq
   	C_{\alpha ,\beta ,\gamma }\lambda (x,y,\theta)^{(2-|
   		\alpha | -|\beta | -|
   		\gamma | )},\quad \forall (x,y,\theta)\in \Omega
   	_{\phi ,\epsilon _{0}}.  \label{3.4}
   	\end{equation}
   	If $|\beta | =1$, then
   	$$
   	|\partial _{x}^{\alpha }\partial _{y }^{\beta }\partial
   	_{\theta}^{\gamma }\phi (x,\theta ,y)| =|\partial
   	_{x}^{\alpha }\partial _{\theta }^{\gamma }(-\theta )
   	| =\begin{cases}
   	0&\text{if }|\alpha | \neq 0 \\
   	|\partial _{\theta }^{\gamma }(-\theta )|&\text{if }
   	\alpha =0;
   	\end{cases}
   	$$
   	If $|\beta | >1$, then $|
   	\partial _{x}^{\alpha }\partial _{y }^{\beta }\partial
   	_{\theta}^{\gamma }\phi (x,y,\theta)| =0$.
   
   	Hence the estimate \eqref{3.4} is satisfied.\\
   	   	If $|\beta | =0$, then for all $\alpha ,\gamma
   	\in \mathbb{N}^n$;
   	$|\alpha |+|\gamma | \leq 2$, there exists $C_{\alpha ,\gamma }>0$
   	such that
   	\begin{equation*}
   	|\partial _{x}^{\alpha }\partial _{\theta }^{\gamma }\phi
   	(x,y,\theta)| =|\partial _{x}^{\alpha
   	}\partial _{\theta }^{\gamma }S(x,\theta )-\partial
   	_{x}^{\alpha }\partial _{\theta }^{\gamma }(y\theta )
   	| \leq C_{\alpha ,\gamma }\lambda (x,y,\theta)^{(2-|\alpha | -|\gamma | )}.
   	\end{equation*}
   	If $|\alpha | +|\gamma | >2$, one has
   	$\partial _{x}^{\alpha }\partial _{\theta }^{\gamma
   	}\phi (x,y,\theta)=\partial _{x}^{\alpha }\partial _{\theta
   }^{\gamma }S(x,\theta )$. In $\Omega _{\phi
   ,\epsilon _{0}}$ we have
\begin{equation*}
|y| =|\partial _{\theta }S(x,\theta )-y
-\partial _{\theta }S(x,\theta )|
\leq \sqrt{\epsilon _{0}}(|x| ^2+|y| ^2
+|\theta | ^2)^{1/2}+C\lambda (x,\theta),
\end{equation*}
with $C>0$.
For $\epsilon _{0}$ sufficiently small, we obtain a constant
$C>0$
such that
\begin{equation}
|y| \leq C\lambda (x,\theta ),\quad
\forall (x,y,\theta)\in \Omega _{\phi ,\epsilon _{0}\;}.
\label{3.5}
\end{equation}
This inequality leads to the equivalence
\begin{equation}
\lambda (x,\theta ,y)\simeq \lambda (x,\theta
)\quad \text{in }\Omega _{\phi ,\epsilon _{0}\;}  \label{3.6}
\end{equation}
thus the assumption $(G3)$ and
\eqref{3.6} give the estimate \eqref{3.4}.
\end{proof}

Using \eqref{3.6}, we have the following result.

\begin{Pro} \label{prop3.5}
	If $(x,\theta )\to a(x,\theta )$
	belongs to $S _{k}^{m}(\mathbb{R}_{x}^n\times
	\mathbb{R}_{\theta }^n)$,  then
	$(x,y,\theta)\to a(x,\theta )$  belongs to
    $S_{k}^{\tilde{m}}(\mathbb{R}_{x}^n\times\mathbb{R}_{y}^n\times \mathbb{R}_{\theta }^n)\cap S _{k}^{\tilde{m}}(\Omega _{\phi
		,\epsilon _{0}\;})$, $k\in \{ 0,1\} $.
\end{Pro}

\section{$L^2$-boundedness and $L^2$-compactness of $F$}	
	The main result is as follows.
	
	\begin{The} \label{thm4.1}
		Let $F$ be the integral operator of distribution kernel
		\begin{equation}
		K(x,y)=\int_{\mathbb{R}^n} e^{i(S(x,\theta )
			-y\theta)}a(x,\theta )\widehat{d\theta }  \label{4.1}
		\end{equation}
		where $\widehat{d\theta }=(2\pi )^{-n}\,d\theta$,
		$a\in S_{k}^{m}(\mathbb{R}_{x,\theta }^{\;2n})$,
		$k=0,1$ and $S$ satisfies $(G1)$, (G2) and (G3).
		Then $FF^{\ast }$  and $F^{\ast }F$ are
		pseudodifferential operators with symbol in
		$ S_{k}^{m^2}(\mathbb{R}^{2n})$, $k=0,1$,
		given by
		\begin{gather*}
			\sigma (FF^{\ast })(x,\partial _{x}S(x,\theta ))
			\equiv |a(x,\theta )| ^2|(\det
			\frac{\partial ^2S}{\partial
				\theta \partial x})^{-1}(x,\theta )| \\
			\sigma (F^{\ast }F)(\partial _{\theta }S(x,\theta ),\theta )
			\equiv |a(x,\theta )| ^2|(\det \frac{\partial ^2S
			}{\partial \theta \partial x})^{-1}(x,\theta )|
		\end{gather*}
		we denote here $a\equiv b$  for
		$a,b\in  S _{k}^{p^2}(\mathbb{R}^{2n})$
		if $(a-b)\in  S _{k}^{p^2\lm^{-2}}(\mathbb{R}^{2n})$
		and $\sigma $ stands for the symbol.
	\end{The}
	
	\begin{proof}
		If $u\in \mathcal{S}(\mathbb{R}^n)$, then $Fu(x)$ is given by
		\begin{equation}
		\begin{aligned}
		Fu(x) &=\int_{\mathbb{R}^n}K(x,y) u(y)\,dy \\
		&=\int_{\mathbb{R}^n}\int_{\mathbb{R}^n}e^{i(S(x,\theta )
			-y\theta)}a(x,\theta )u(y)dy\widehat{d\theta }  \\
		&=\int_{\mathbb{R}^n}e^{iS(x,\theta )} a(x,\theta
		)\Big(\int_{\mathbb{R}^n}e^{-iy\theta}u(y)dy\Big)\widehat{d\theta }
		\\
		&=\int_{\mathbb{R}^n}e^{iS(x,\theta )} a(x,\theta ) \mathcal{F}
		u(\theta )\widehat{d\theta }.
		\end{aligned}  \label{4.2}
		\end{equation}
		Here $F$ is a continuous linear mapping from
		$\mathcal{S}(\mathbb{R}^n)$ to
		$\mathcal{S}(\mathbb{R}^n)$ (by Theorem \ref{existence}).
		Let $v\in \mathcal{S}(\mathbb{R}^n)$, then
		\begin{align*}
			\langle Fu,v\rangle_{L^2(\mathbb{R}^n)}
			&=\int_{\mathbb{R}^n}\Big(\int_{
				\mathbb{R}^n}e^{iS(x,\theta )}a(x,\theta )\mathcal{F}u(\theta )
			\widehat{d\theta }\Big) \overline{v(x)}\,dx \\
			&=\int_{\mathbb{R}^n}\mathcal{F}u(\theta )\Big(\int_{\mathbb{R}
				^n}\overline{e^{-iS(x,\theta )} \overline{a(x,\theta )}v(x)\,dx}\Big)
			\widehat{d\theta }
		\end{align*}
		thus
		\begin{equation*}
			\langle Fu(x),v(x)\rangle _{L^2(\mathbb{R}^n)}
			=(2\pi )^{-n}\langle \mathcal{F}u(\theta ),\mathcal{F}((
			F^{\ast }v))(\theta )\rangle_{L^2(\mathbb{R}^n)}
		\end{equation*}
		where
		\begin{equation}
		\mathcal{F}((F^{\ast }v))(\theta)
		=\int_{\mathbb{R}^n}e^{-iS(\widetilde{x},\theta )}\overline{a}
		(\widetilde{x},\theta )
		v(\widetilde{x})d\widetilde{x}. \label{4.3}
		\end{equation}
		Hence, for all $v\in \mathcal{S}(\mathbb{R}^n)$,
		\begin{equation}
		(FF^{\ast }v)(x)=\int_{\mathbb{R}
			^n}\int_{\mathbb{R}^n}e^{i(S(x,\theta
			)-S(\widetilde{x},\theta ))}a(x,\theta )
		\overline{a}(\widetilde{x},\theta )d\widetilde{x}
		\widehat{d\theta }.  \label{4.4}
		\end{equation}
		The main idea to show that $FF^{\ast }$ is a pseudodifferential
		operator, is to use the fact that
		$(S(x,\theta )-S(\widetilde{x},\theta ))$ can be expressed by
		the scalar product
		$\langle x-\widetilde{x},\xi (x,\widetilde{x},\theta )\rangle$ after
		considering the change of variables
		$(x,\widetilde{x},\theta)\to (x,\widetilde{x},\xi
		=\xi (x,\widetilde{x},\theta ))$.
		
		The distribution kernel of $FF^{\ast }$ is
		\begin{equation*}
			K(x,\tilde{x})=\int_{\mathbb{R}^n}e^{i(
				S(x,\theta )-S(\tilde{x},\theta ))}a(x,\theta
			)\overline{a}(\tilde{x},\theta )\widehat{d\theta }.
		\end{equation*}
		We obtain from \eqref{3.3} that if
		$|x-\widetilde{x}| \geq \frac{\epsilon
		}{2}\lambda (x,\widetilde{x},\theta )$
		(where $\epsilon >0$ is sufficiently small)
		then
		\begin{equation}
		|(\partial _{\theta }S)(x,\theta )-(
		\partial _{\theta }S)(\widetilde{x},\theta )
		| \geq \frac{\epsilon }{2C}\lambda (
		x,\widetilde{x},\theta ). \label{4.5}
		\end{equation}
		Choosing $\omega \in C^{\infty }(\mathbb{R})$ such
		that
		\begin{gather*}
			\omega (x)\geq 0, \quad \forall x\in \mathbb{R}   \\
			\omega (x)=1 \quad  \text{if }  x\in [ -\frac{1}{2},\frac{1}{2}] \\
			\mathop{\rm supp}\omega \subset  ] -1,1[
		\end{gather*}
		and setting
		\begin{gather*}
			b(x,\tilde{x},\theta ):= a(x,\theta )\overline{a}
			(\tilde{x},\theta )=b_{1,\epsilon }(
			x,\tilde{x},\theta )+b_{2,\epsilon }(x,\tilde{x},\theta )
			\\
			b_{1,\epsilon }(x,\tilde{x},\theta )=\omega (\frac{
				|x-\tilde{x}| }{\epsilon \lambda (x,\tilde{x}
				,\theta )})b(x,\tilde{x},\theta )
			\\
			b_{2,\epsilon }(x,\tilde{x},\theta )=[
			1-\omega (\frac{|x-\tilde{x}| }{\epsilon \lambda (x,
				\tilde{x},\theta )})] b(x,\tilde{x},\theta ).
		\end{gather*}
		We have $K(x,\widetilde{x})=K_{1,\epsilon }(x,
		\widetilde{x})+K_{2,\epsilon }(
		x,\widetilde{x})$, where
		\begin{equation*}
			K_{j,\epsilon }(x,\tilde{x})=\int_{\mathbb{R}
				^n}e^{i(S(x,\theta )-S(\tilde{x},\theta ))
			}b_{j,\epsilon }(x,\tilde{x},\theta )
			\widehat{d\theta },\quad j=1,2.
		\end{equation*}
		We will study separately the kernels $K_{1,\epsilon }$ and
		$K_{2,\epsilon } $.
		
		On the support of $b_{2,\epsilon }$, inequality \eqref{4.5}
		is satisfied and we have
		\begin{equation*}
			K_{2,\epsilon }(x,\widetilde{x})\in
			\mathcal{S}(\mathbb{R}^n\times \mathbb{R}^n).
		\end{equation*}
		Indeed, using the oscillatory integral method, there is a linear
		partial differential operator $L$ of order 1 such that
		\begin{equation*}
			L\big(e^{i(S(x,\theta )-S(\tilde{x},\theta ))
			}\big)=e^{i(S(x,\theta )-S(\tilde{x},\theta ))}
		\end{equation*}
		where
		\begin{equation*}
			L=-i|(\partial _{\theta }S)
			(x,\theta )-(\partial _{\theta }S)(
			\widetilde{x},\theta )|
			^{-2}\sum_{l=1}^n\left[ (\partial _{\theta
				_{l}}S)(x,\theta )-(\partial _{\theta _{l}}S)
			(\widetilde{x},\theta )\right] \partial _{\theta_{l}}.
		\end{equation*}
		The transpose operator of $L$ is
		\begin{equation*}
			^{t}L=\sum_{l=1}^nF_{l}(x,\widetilde{x},\theta)
			\partial _{\theta _{l}}+G(x,\widetilde{x},\theta )
		\end{equation*}
		where $F_{l}(x,\widetilde{x},\theta )\in  S
		_{0}^{\lambda^{-1}}(\Omega _{\epsilon })$,
		$G(x,\widetilde{x},\theta )\in  S _{0}^{\lambda^{-2}}(\Omega
		_{\epsilon })$,
		\begin{gather*}
			F_{l}(x,\widetilde{x},\theta )=i|(
			\partial _{\theta }S)(x,\theta )-(\partial _{\theta
			}S)(\widetilde{x},\theta )|
			^{-2}((\partial _{\theta _{l}}S)(x,\theta
			)-(\partial _{\theta _{l}}S)
			(\widetilde{x},\theta )),
			\\
			G(x,\widetilde{x},\theta )
			=i\sum_{l=1}^n\partial _{\theta _{l}}\left[ |
			(\partial _{\theta }S)(x,\theta )-(\partial
			_{\theta }S)(\widetilde{x},\theta )|
			^{-2}((\partial _{\theta _{l}}S)(x,\theta)-(\partial _{\theta _{l}}S)(
			\widetilde{x},\theta )
			)\right],
			\\
			\Omega _{\epsilon }=\big\{ (x,\tilde{x},\theta )\in
			\mathbb{R}^{3n}:|\partial _{\theta }S(x,\theta )-\partial
			_{\theta}S(\tilde{x},\theta )| >\frac{\epsilon }{2C}
			\lambda (x,\tilde{x},\theta )\big\} .
		\end{gather*}
		On the other hand we prove by induction on $q$ that
		\begin{equation*}
			(^{t}L)^{q}b_{2,\epsilon }(x,\tilde{x},\theta )
			=\sum_{|\gamma | \leq q ,\, \gamma \in \mathbb{N}^n}
			g_{\gamma ,q}(x,\tilde{x},\theta )\partial
			_{\theta }^{\gamma }b_{2,\epsilon }(x,\tilde{x},\theta ),
			\text{ }g_{\gamma }^{(q)}\in S _{0}^{\lambda^{-q}}(
			\Omega _{\epsilon }),
		\end{equation*}
		and so,
		\begin{equation*}
			K_{2,\epsilon }(x,\tilde{x})=\int_{\mathbb{R}
				^n}e^{i(S(x,\theta )-S(\tilde{x},\theta ))}(
			^{t}L)^{q}b_{2,\epsilon }(x,\tilde{x},\theta
			)\widehat{d\theta }.
		\end{equation*}
		
		Using Leibniz's formula, (G3) and the form
		$(^{t}L)^{q}$, we can choose $q$ large enough such that
		for all $\alpha ,\alpha ',\beta ,\beta '\in \mathbb{N}
		^n,\exists C_{\alpha ,\alpha ',\beta ,\beta '}>0$,
		\begin{equation*}
			\sup_{x,\widetilde{x}\in \mathbb{R}^n}|x^{\alpha }\widetilde{x}
			^{\alpha '}\partial _{x}^{\beta }\partial
			_{\widetilde{x}}^{\beta '}K_{2,\epsilon }(
			x,\widetilde{x})| \leq C_{\alpha ,\alpha ',\beta ,\beta '}.
		\end{equation*}
		
		Next, we study $K_{1}^{\epsilon }$: this is more difficult and
		depends on the choice of the parameter $\epsilon $. It follows
		from Taylor's formula that
		\begin{gather*}
			S(x,\theta )-S(\widetilde{x},\theta )
			=\langle x- \widetilde{x},\xi (x,\widetilde{x},\theta )
			\rangle_{\mathbb{R}^n} ,\\
			\xi (x,\widetilde{x},\theta )
			=\int_{0}^{1}(\partial _{x}S)(\widetilde{x}+t(x-\widetilde{x})
			,\theta )dt.
		\end{gather*}
		We define the vectorial function
		\begin{equation*}
			\widetilde{\xi }_{\epsilon }(x,\widetilde{x},\theta
			)=\omega \big(\frac{|x-\tilde{x}|
			}{2\epsilon \lambda (x,\tilde{x},\theta )}\big)
			\xi (x,\widetilde{x},\theta
			)+\big(1-\omega (\frac{|x-\tilde{x}| }{
				2\epsilon \lambda (x,\tilde{x},\theta )})
			\big)(\partial _{x}S)(\widetilde{x},\theta).
		\end{equation*}
		We have
		\begin{equation*}
			\widetilde{\xi }_{\epsilon }(x,\widetilde{x},\theta
			)=\xi (x,\widetilde{x},\theta )\text{ on }
			\mathop{\rm supp} b_{1,\epsilon }.
		\end{equation*}
		Moreover, for $\epsilon $ sufficiently small,
		\begin{equation}
		\lambda (x,\theta )\simeq \lambda ( \widetilde{x},\theta )\simeq
		\lambda ( x,\widetilde{x},\theta )\text{ on } \mathop{\rm supp}
		b_{1,\epsilon }. \label{4.6}
		\end{equation}
		Let us consider the mapping
		\begin{equation}
		\mathbb{R}^{3n}\ni (x,\widetilde{x},\theta )
		\to (x,\widetilde{x},\widetilde{\xi }_{\epsilon
		}(x,\widetilde{x},\theta ))\label{4.7}
		\end{equation}
		for which Jacobian matrix is
		\begin{equation*}
			\begin{pmatrix}
				I_{n} & 0 & 0 \\
				0 & I_{n} & 0 \\
				\partial _{x}\widetilde{\xi }_{\epsilon } & \partial _{\widetilde{x}}
				\widetilde{\xi }_{\epsilon } & \partial _{\theta }\widetilde{\xi }
				_{\epsilon }
			\end{pmatrix}.
		\end{equation*}
		We have
		\begin{align*}
			&\frac{\partial \widetilde{\xi }_{\epsilon ,j}}{\partial \theta _{i}}
			(x,\widetilde{x},\theta )\\
			&=\frac{\partial ^2S}{\partial \theta
				_{i}\partial x_{j}}(\widetilde{x},\theta )
			+\omega \big(\frac{
				|x-\tilde{x}| }{2\epsilon \lambda (x,\tilde{x}
				,\theta )}\big)\big(\frac{\partial \xi _{j}}{\partial \theta _{i}
			}(x,\widetilde{x},\theta )-\frac{\partial
			^2S}{\partial
			\theta _{i}\partial x_{j}}(\widetilde{x},\theta )\big)
		\\
		&\quad -\frac{|x-\tilde{x}| }{2\epsilon \lambda (x,
			\tilde{x},\theta )}\frac{\partial \lambda }{\partial \theta _{i}}
		(x,\tilde{x},\theta )\lambda ^{-1}(
		x,\tilde{x},\theta
		)\omega '\big(\frac{|x-\tilde{x}| }{
			2\epsilon \lambda (x,\tilde{x},\theta )}\big)
		\big(\xi_{j}(x,\widetilde{x},\theta )-\frac{\partial S}{\partial x_{j}}
		(\widetilde{x},\theta )\big).
	\end{align*}
	Thus, we obtain
	\begin{align*}
		&\big|\frac{\partial \widetilde{\xi }_{\epsilon
				,j}}{\partial \theta _{i}}(x,\widetilde{x},\theta )
		-\frac{\partial ^2S}{\partial \theta _{i}\partial x_{j}}(
		\widetilde{x},\theta )\big|\\
		&\leq \big|\omega (\frac{|x-\tilde{x}| }{
			2\epsilon \lambda (x,\tilde{x},\theta )})\big|
		\big|\frac{\partial \xi _{j}}{\partial \theta _{i}}(x,\widetilde{
			x},\theta )-\frac{\partial ^2S}{\partial \theta _{i}\partial x_{j}}
		(\widetilde{x},\theta )\big| \\
		&\quad + \lambda ^{-1}(x,\tilde{x},\theta )\big|\omega'(\frac{|x-\tilde{x}|
		}{2\epsilon \lambda (x,\tilde{x},\theta )})
		\big| \big|\xi_{j}(x,\widetilde{x},\theta )
		-\frac{\partial S}{\partial x_{j}}
		(\widetilde{x},\theta )\big| .
	\end{align*}
	Now it follows from (G3), \eqref{4.6}
	and Taylor's formula that
	\begin{equation}
	\begin{aligned}
	\big|\frac{\partial \xi _{j}}{\partial \theta _{i}}(x,\widetilde{
		x},\theta )-\frac{\partial ^2S}{\partial \theta _{i}\partial x_{j}}
	(\widetilde{x},\theta )\big|
	&\leq \int_{0}^{1}\big|\frac{\partial ^2S}{\partial
		\theta _{i}\partial x_{j}}(\widetilde{x}+t(
	x-\widetilde{x})
	,\theta )-\frac{\partial ^2S}{\partial \theta _{i}\partial x_{j}}
	(\widetilde{x},\theta )\big| dt
	\\
	&\leq C|x-\widetilde{x}| \lambda ^{-1}(x,
	\tilde{x},\theta ),\quad C>0
	\end{aligned} \label{4.8}
	\end{equation}
	
	\begin{equation}
	\begin{aligned}
	\big|\xi _{j}(x,\widetilde{x},\theta )-\frac{\partial S}{
		\partial x_{j}}(\widetilde{x},\theta )\big|
	&\leq \int_{0}^{1}\big|\frac{\partial S}{\partial
		x_{j}}(\widetilde{x}+t(x-\widetilde{x}),\theta
	)-\frac{\partial S}{\partial x_{j}}(
	\widetilde{x},\theta )\big| dt
	\\
	&\leq C|x-\widetilde{x}| ,\quad C>0\,.\label{4.9}
	\end{aligned}
	\end{equation}
	From \eqref{4.8} and \eqref{4.9},
	there exists a positive constant $C>0$ such that
	\begin{equation}
	|\frac{\partial \widetilde{\xi }_{\epsilon
			,j}}{\partial \theta _{i}}(x,\widetilde{x},\theta )
	-\frac{\partial ^2S}{\partial \theta _{i}\partial x_{j}}(
	\widetilde{x},\theta )| \leq C\epsilon
	,\quad \forall i,j\in \{ 1,\dots ,n\} . \label{4.10}
	\end{equation}
	If $\epsilon <\frac{\delta _{0}}{2\widetilde{C}}$, then
	\eqref{4.10} and (G2) yields the estimate
	\begin{equation}
	\delta _{0}/2\leq -\widetilde{C}\epsilon +\delta _{0}\leq -\widetilde{C}
	\epsilon +\det \frac{\partial ^2S}{\partial x\partial \theta
	}(x,\theta
	)\leq \det \partial _{\theta }\widetilde{\xi }_{\epsilon }(x,
	\widetilde{x},\theta ),
	\label{4.11}
	\end{equation}
	with $\widetilde{C}>0$.
	If $\epsilon $ is such that \eqref{4.6} and \eqref{4.11} hold,
	then the mapping given in \eqref{4.7} is a global diffeomorphism of
	$\mathbb{R}^{3n}$. Hence there exists a
	mapping
	\begin{equation*}
		\theta :\mathbb{R}^n\times \mathbb{R}^n\times \mathbb{R}^n\ni (x,
		\widetilde{x},\xi )\to \theta (
		x,\widetilde{x},\xi )\in \mathbb{R}^n
	\end{equation*}
	such that
	\begin{equation}
	\begin{gathered}
	\widetilde{\xi }_{\epsilon }(x,\widetilde{x},\theta (x,
	\widetilde{x},\xi ))=  \xi   \\
	\theta (x,\widetilde{x},\widetilde{\xi }_{\epsilon }(x,
	\widetilde{x},\theta ))=  \theta   \\
	\partial ^{\alpha }\theta (x,\widetilde{x},\xi )=\mathcal{O}
	(1), \quad  \forall \alpha \in
	\mathbb{N}^{3n}\backslash \{0\}
	\end{gathered} \label{4.12}
	\end{equation}
	If we change the variable $\xi $ by
	$\theta (x,\widetilde{x},\xi)$ in $K_{1,\epsilon }(x,\widetilde{x})$,
	we obtain
	\begin{equation}
	K_{1,\epsilon }(x,\widetilde{x})=\int_{\mathbb{R}
		^n}e^{i\langle x-\tilde{x},\xi \rangle}b_{1,\epsilon }(
	x,\tilde{x},\theta (x,\widetilde{x},\xi ))
	\big|\det \frac{\partial \theta }{\partial \xi }(
	x,\widetilde{x},\xi )\big| \widehat{d\xi }.
	\label{4.13}
	\end{equation}
	From \eqref{4.12} we have, for $k=0,1$, that
	$b_{1,\epsilon }(x,\tilde{x},\theta (
	x,\widetilde{x},\xi ))|\det \frac{\partial
		\theta }{\partial \xi }(x,\widetilde{x},\xi )| $ belongs to
	$S _{k}^{m^2}(\mathbb{R}^{3n})$ if
	$a\in  S _{k}^{m}(\mathbb{R}^{2n})$.
	
	Applying the stationary phase theorem (c.f. \cite{Robert} ) to
	\ref{4.13}, we obtain the expression of the symbol of the
	pseudodifferential
	operator $FF^{\ast }$,
	\begin{equation*}
		\sigma (FF^{\ast })=b_{1,\epsilon }(x,\tilde{x},\theta (x,
		\widetilde{x},\xi ))\big|\det \frac{\partial \theta }{
			\partial \xi }(x,\widetilde{x},\xi )\big|
		_{|\widetilde{x}=x }+R(x,\xi )
	\end{equation*}
	where $R(x,\xi )$ belongs to
	$ S _{k}^{m^2\lm^{-2}}(\mathbb{R}^{2n})$ if
	$a\in  S _{k}^{m}(\mathbb{R}^{2n})$, $k=0,1$.
	
	For $\tilde{x}=x$, we have
	$b_{1,\epsilon }(x,\tilde{x},\theta (x,\widetilde{x},\xi ))
	=|a(x,\theta (x,x,\xi ))| ^2$ where $\theta (x,x,\xi )$ is the
	inverse of the mapping
	$\theta \to \partial_{x}S(x,\theta )=\xi $. Thus
	\begin{equation*}
		\sigma (FF^{\ast })(x,\partial _{x}S(x,\theta )
		)\equiv |a(x,\theta )|
		^2\big|\det \frac{\partial ^2S}{\partial \theta \partial
			x}(x,\theta )\big| ^{-1}.
	\end{equation*}
	From \eqref{4.2} and \eqref{4.3}, we
	obtain the expression of $F^{\ast }F$:
	$\forall v\in \mathcal{S}(\mathbb{R}^n)$,
	\begin{align*}
		(\mathcal{F(}F^{\ast }F)\mathcal{F}^{-1})v(
		\theta )
		&=\int_{\mathbb{R}^n}e^{-iS(x,\theta )}\overline{a}
		(x,\theta )(F(\mathcal{F}^{-1}v))(x)dx \\
		&=\int_{\mathbb{R}^n}e^{-iS(x,\theta )}\overline{a}
		(x,\theta )\Big(\int_{\mathbb{R}^n}e^{iS(x,\widetilde{\theta }
			)}a(x,\widetilde{\theta })(\mathcal{F}(\mathcal{F}
		^{-1}v))(\widetilde{\theta })
		\widehat{d\widetilde{\theta }}\Big)dx\newline
		\\
		&=\int_{\mathbb{R}^n}\int_{\mathbb{R}^n}e^{-i(
			S(x,\theta )-S(x,\tilde{\theta}))\;}
		\overline{a}(x,\theta )\;a(x,\widetilde{\theta })v(\tilde{
			\theta})\widehat{d\widetilde{\theta }}dx\text{.}
	\end{align*}
	Hence the distribution kernel of the integral operator
	$\mathcal{F(}F^{\ast }F)\mathcal{F}^{-1}$ is
	\begin{equation*}
		\widetilde{K}(\theta ,\widetilde{\theta })
		=\int_{\mathbb{R}^n}e^{-i(S(x,\theta )-S(x,\tilde{\theta})
			)}\overline{a}(x,\theta ) a(x,\tilde{\theta})\widehat{dx}.
	\end{equation*}
	We remark that we can deduce
	$\widetilde{K}(\theta ,\widetilde{\theta })$ from
	$K(x,\widetilde{x})$ by replacing $x$ by $\theta $. On the other
	hand, all
	assumptions used here are symmetrical on $x$ and $\theta $; therefore,
	$\mathcal{F(}F^{\ast }F)\mathcal{F}^{-1}$ is a nice
	pseudodifferential
	operator with symbol
	\begin{equation*}
		\sigma (\mathcal{F(}F^{\ast }F)\mathcal{F}^{-1})(\theta
		,-\partial _{\theta }S(x,\theta ))\equiv |
		a(x,\theta )|
		^2\big|\det \frac{\partial ^2S}{\partial x\partial \theta }
		(x,\theta )\big| ^{-1}.
	\end{equation*}
	Thus the symbol of $F^{\ast }F$ is given by (c.f. \cite{Hor-Weyl})
	\begin{equation*}
		\sigma (F^{\ast }F)(\partial _{\theta }S(x,\theta ),\theta )\equiv
		|a(x,\theta )| ^2\big|\det \frac{\partial ^2S}{
			\partial x\partial \theta }(x,\theta )\big| ^{-1}.
	\end{equation*}
\end{proof}

\begin{corollary} \label{coro4.2}
	Let $F$ be the integral operator with the
	distribution kernel
	\begin{equation*}
		K(x,y)=\int_{\mathbb{R}^n} e^{i(S(x,\theta )-y\theta)}
		a(x,\theta )\widehat{d\theta }
	\end{equation*}
	where $a\in  S _{0}^{m}(\mathbb{R}_{x,\theta }^{2n})$
	and $S$  satisfies (G1), (G2) and (G3).
	Then, we have:
	\begin{enumerate}
		\item For any bounded tempered weight $m$, $F$ can be
		extended as a bounded linear mapping on $L^2(\mathbb{R}^n)$
		\item  For any $m$ such that
		$\underset{|x|+|\theta|\rightarrow\infty}\lim m(x,\theta)=0$, $F$ can be extended as a compact operator on
		$L^2(\mathbb{R}^n)$.
	\end{enumerate}
\end{corollary}

\begin{proof}
	It follows from Theorem \ref{thm4.1} that $F^{\ast }F$ is a
	pseudodifferential operator with symbol in
	$ S _{0}^{m^2}(\mathbb{R}^{2n})$.
	
	\noindent (1) Since $m$ is bounded, we can apply the Cald\'{e}ron-Vaillancourt theorem
	(see \cite{Cal-Vai}) for
	$F^{\ast }F$ and obtain the existence of a positive constant
	$\gamma (n)$ and an integer $k(n)$ such that
	\begin{equation*}
		\|(F^{\ast }F)\;u\|_{L^2(\mathbb{R}^n)}\leq
		\gamma (n)\;Q_{k(n)}(\sigma (FF^{\ast }))
		\|u\|_{L^2(\mathbb{R}^n)},\quad \forall
		u\in \mathcal{S}(\mathbb{R}^n)
	\end{equation*}
	where
	\begin{equation*}
		Q_{k(n)}(\sigma (FF^{\ast }))
		=\sum_{|\alpha | +|\beta | \leq k(n)}\sup_{(x,\theta )\in
			\mathbb{R}^{2n}} \big|\partial _{x}^{\alpha }\partial
		_{\theta }^{\beta }\sigma (FF^{\ast })(\partial _{\theta
		}S(x,\theta ),\theta )\big|
	\end{equation*}
	Hence,  for all $u\in \mathcal{S}(\mathbb{R}^n)$,
	\begin{equation*}
		\|Fu\|_{L^2(\mathbb{R}^n)}\leq \|
		F^{\ast }F\|_{_{\mathcal{L}(
				L^2(\mathbb{R}^n))}}^{1/2}\|u\|
		_{L^2(\mathbb{R}^n)}\leq (\gamma (n)\;Q_{k(
			n)}(\sigma (FF^{\ast })))
		^{1/2}\|u\|_{L^2(\mathbb{R}^n)}.
	\end{equation*}
	Thus $F$ is also a bounded linear operator on
	$L^2(\mathbb{R}^n)$.
	
	\noindent (2) If $\underset{|x|+|\theta|\rightarrow\infty}\lim m(x,\theta)=0$, the compactness theorem
	(see \cite{Robert}) shows that the operator $F^{\ast }F$ can be
	extended as a compact operator on $L^2(\mathbb{R}^n)$.
	Thus, the Fourier integral operator
	$F$ is compact on $L^2(\mathbb{R}^n)$.
	Indeed, let $(\varphi _{j})_{j\in \mathbb{N}}$ be an orthonormal basis of
	$L^2(\mathbb{R}^n)$, then
	\begin{equation*}
		\|F^{\ast }F-\sum_{j=1}^n \langle \varphi_{j},.\rangle F^{\ast }
		F\varphi _{j}\| \to 0 \quad \text{as } n\to +\infty.
	\end{equation*}
	Since $F$ is bounded, for all $\psi \in L^2(\mathbb{R}^n)$,
	\[
	\big\|F\psi -\sum_{j=1}^n \langle \varphi_{j},\psi
	\rangle F\varphi _{j}\big\|^2
	\leq \big\|F^{\ast }F\psi -\sum_{j=1}^n \langle \varphi
	_{j},\psi \rangle F^{\ast }F\varphi _{j}\big\|
	\big\|\psi -\sum_{j=1}^n \langle \varphi _{j},\psi \rangle
	\varphi _{j}\big\|,
	\]
	it follows that
	\begin{equation*}
		\|F-\sum_{j=1}^n \langle \varphi
		_{j},.\rangle F\varphi _{j}\| \to 0 \quad \text{as } n\to +\infty
	\end{equation*}
\end{proof}

\begin{example} \label{exa4.3} 
	We consider the function given by
	\begin{equation*}
		S(x,\theta )=\sum_{|\alpha| +|\beta | =2,\,
			\alpha ,\beta \in \mathbb{N}^n}
		C_{\alpha ,\beta }x^{\alpha }\theta ^{\beta },\quad
		\text{for }(x,\theta )\in \mathbb{R}^{2n}
	\end{equation*}
	where $C_{\alpha ,\beta }$ are real constants. This function satisfies
	(G1), (G2) and (G3).  
\end{example}
\bibliography{bib}
\nocite{FIOI,Me-Se,Cal-Vai,shubin,Hor-Weyl,Se-h-adim}

\end{document}